\documentclass[10pt]{amsart}
\usepackage{amsmath, amsthm,amsxtra}
\usepackage[english]{babel}
\usepackage[T1]{fontenc}
\usepackage[latin1]{inputenc}
\usepackage{amssymb}
\usepackage{amscd}
\usepackage{latexsym}
\usepackage{graphicx}
\usepackage{mathrsfs} % corsivo
\usepackage[all,cmtip]{xy}

\usepackage{verbatim}
\newtheoremstyle{theorem}%name
  {12pt}          % space above
  {12pt}  % space below
  {\sl}  % bofy font
  {\parindent}     % ident - empty=no indent,  \parindent= paragraph indent
  {\bf}  % thm head font
  {. }    % punctuation after thm head
  { }    % space after thm head: `` ``=normal \newline=linebreak
  {}     % thm head specification
\theoremstyle{theorem}
\newtheorem{theorem}{Theorem}

\newtheorem{remark}[theorem]{Remark}
\newtheorem{proposition}[theorem]{Proposition}
\newtheorem{lemma}[theorem]{Lemma}

\newtheorem{definition}[theorem]{Definition}

\linethickness{0.5pt}

\newcommand{\ic}{\ensuremath{\mathcal{I}}}

\newcommand{\oc}{\ensuremath{\mathcal{O}}}

\newcommand{\fc}{\ensuremath{\mathcal{F}}}

\newcommand{\ec}{\ensuremath{\mathcal{E}}}

\newcommand{\Pt}{\mathbb{P}^3}

\newcommand{\Ptw}{\mathbb{P}^2}

\newcommand{\cG}{\gamma}

\newcommand{\lG}{\lambda}

\newcommand{\fG}{\varphi}

\newcommand{\bds}{\begin{displaystyle}}
\newcommand{\eds}{\end{displaystyle}}

\begin{document}
\title[Global Tjurina number]{Quasi-complete intersections and global Tjurina number of plane curves.}

\author{Ph. Ellia}
\address{Dipartimento di Matematica e Informatica, Universit\`a degli Studi di Ferrara, Via Machiavelli 30, 44121 Ferrara, Italy.}
\email{phe@unife.it}

\subjclass[2010] {Primary 14H50; Secondary 14M06, 14M07, 13D02} \keywords{quasi complete intersections, codimension two, vector bundle, global Tjurina number, plane curves.}

\begin{abstract} A closed subscheme of codimension two $T \subset \Ptw$ is a quasi complete intersection (q.c.i.) of type $(a,b,c)$ if there exists a surjective morphism $\oc (-a) \oplus \oc (-b) \oplus \oc (-c) \to \ic _T$. We give bounds on $\deg (T)$ in function of $a,b,c$ and $r$, the least degree of a syzygy between the three polynomials defining the q.c.i. (see Theorem \ref{T-the}). As a by-product we recover a theorem of du Plessis-Wall on the global Tjurina number of plane curves (see Theorem \ref{T-dP-Wall}) and some other related results. 
\end{abstract}

\date{\today}

%\begin{document}
%%%%%%%%%%%%%%%%%%%%%%%%%%%%%%%%%%%%%%%%%%%%%%%%%%%%%%%%%
\maketitle

%\tableofcontents

\thispagestyle{empty}

\section{Introduction.}

Let $C \subset \Ptw$ be a reduced, singular curve, of degree $d$, of equation $f=0$. The partials of $f$ determine a morphism: $3.\oc \stackrel{\partial f}{\to} \oc (d-1)$, whose image is a, twisted, ideal sheaf, $\ic _f(d-1)$. From the assumption on $C$, it follows that $\ic _f$ is the ideal sheaf of a zero-dimensional subscheme, $\Sigma$, of $C$, called the \emph{jacobian scheme of $C$}. (Hence $\ic _f = \ic _\Sigma$.) The support of $\Sigma$ is the singular locus of $C$, but the scheme structure is rather mysterious. The \emph{global Tjurina number of $C$} is $\tau (C) := h^0(\oc _\Sigma )$ (we will often write $\tau$ instead of $\tau (C)$ if no confusion can arise).

It is natural to ask for some bound on $\tau$ in function of $d$ (and of some other natural invariants). 

The kernel of the morphism $\partial f$ is a rank two reflexive sheaf, $E_C$. Since we are on a smooth surface, $E_C$ is in fact locally free. Let $r$ denotes the minimal twist of $E_C$ having a section. In other words $r$ is the least degree of a syzygy between the partials of $f$. Then a very nice result of du Plessis-Wall (see Theorem \ref{T-dP-Wall}) gives bounds on $\tau$ in function of $d$ and $r$. This result is proved in the framework of singularity theory. The proof is hard to follow for those who, like me, are far from being experts in this field. So I tried to find another proof. It turns out that the theorem of du Plessis and Wall is a direct consequence of a general statement about quasi-complete intersections in $\Ptw$ (see Theorem \ref{T-the}), whose proof requires only notions of projective geometry (vector bundles, liaison). So, at this point, one doesn't even need to know what is a partial to prove du Plessis-Wall's theorem ! This is amazing and the first reaction is to think, that using the specific assumption (i.e. the three polynomials giving the quasi-complete intersection are the partials of a single polynomial $f$), one could improve the bounds given by the theorem. Alas this is not the case, examples (see \ref{P- bds sharp}) show that, in some sense, du Plessis and Wall's theorem is sharp. This is the content of the first sections of this paper. In the last section, I give some related results, some known, some new, but all in the framework of projective geometry.

%^^^^^^^^^^^^^^^^^^^^^^^^^^^^^^^^^^

\section{Quasi-complete intersections of codimension two in $\Ptw$.}

Let us start with a definition:

\begin{definition} Let $F_a, F_b, F_c \in S := k[x_0, ..., x_n]$ be three homogeneous polynomials of degrees $a, b, c$. The ideal $J = (F_a, F_b, F_c)$ is said to be a \emph{quasi complete intersection} (q.c.i.) if the morphism $\oc (-a) \oplus \oc (-b) \oplus \oc (-c) \stackrel{\rho}{\to} \oc$, defined by these polynomials, has for image the ideal sheaf, $\ic _T$, of a codimension two subscheme $T$.
\end{definition}

\begin{remark} Sometimes one says, and we will do it, that the subscheme $T$ is a q.c.i. of type $(a,b,c)$ if there exists a surjective morphism $\oc (-a) \oplus \oc (-b) \oplus \oc (-c) \stackrel{\rho}{\to} \ic _T$. Observe that $T$ does not determine $(a,b,c)$. For example a point in $\Ptw$ is q.c.i. $(a,b,c)$ for any $c \geq b \geq a \geq 1$. However if $n \geq 3$ and if $T$ is locally a complete intersection (l.c.i.), then it is true that $T$ determines $(a,b,c)$, see Proposition 1 of \cite{BeoE2}.
\end{remark}

The kernel of a surjective morphism $\oc (-a) \oplus \oc (-b) \oplus \oc (-c) \stackrel{\rho}{\to} \ic _T$ is a rank two reflexive sheaf (Prop. 1 of \cite{Ha}), $\fc$. Clearly the graded module $H^0_*(\fc )$ is the module of syzygies between $F_a, F_b$ and $F_c$.

\begin{definition} The q.c.i. $T$ (or better the ideal $J= (F_a, F_b, F_c))$ is said to be an \emph{almost complete intersection} (a.c.i.) if $\fc$ splits: $\fc = \oc (-p) \oplus \oc (-m)$ (i.e. the module $H^0_*(\fc )$ is free).
\end{definition} 

\begin{remark} So $J$ is an a.c.i. if and only if it is saturated ($J = J^{sat} = H^0_*(\ic _T)$).

In terms of $T$ the definition can be unfortunate: one expect a c.i. to be an a.c.i. If $T$ is a point in $\Ptw$ and if $c \geq b \geq 2$, then $T$ is q.c.i. of type $(a,b,c)$ but the corresponding $\fc$ never splits ($J$ is not saturated). So the c.i. $T$ yields a q.c.i. which is not an a.c.i., this can be confusing.

Observe, by the way, that this cannot happen if $n>2$. If $T$ is a c.i. then $\fc$ splits by Horrocks' theorem. (If $n \geq 4$ any q.c.i. $T$ which is integral and subcanonical is a c.i. \cite{BeoE1}).
\end{remark}

From now on we will assume $n = 2$. In this case, since we are on a smooth surface, reflexive implies locally free so $\fc$ is a rank two vector bundle.

\begin{lemma}
\label{L-Chern}
Let $J = (F_a, F_b, F_c) \subset S = k[x,y,z]$ be a q.c.i. ideal of type $(a,b,c)$, $a \leq b \leq c$. So we have an exact sequence:
\begin{equation}
\label{eq.def qci}
0 \to E \to \oc (c-a) \oplus \oc (c-b) \oplus \oc \to \ic _T(c) \to 0
\end{equation}
where $T \subset \Ptw$ is a closed subscheme of codimension two and where $E$ is a rank two vector bundle. 

Then we have $c_1(E) = c-a-b$ and $c_2(E) = ab -t$, where $t := \deg (T) = h^0(\oc _T)$.
\end{lemma}

\begin{proof} If $T \subset \Ptw$ is zero-dimensional, the Chern classes of $\oc _T$ are $(0, -t)$. This can be seen by starting with one point (see Section 2 of \cite{Ha} for similar computations). From the exact sequence $0 \to \ic _T(c) \to \oc (c) \to \oc _T \to 0$, we get that the Chern classes of $\ic _T(c)$ are $(c, t)$. We conclude with the exact sequence (\ref{eq.def qci}).
\end{proof}

For later use we have twisted the previous exact sequence by $c$, so $E = \fc (c)$ with the previous notations.

Here is the main result:

\begin{theorem}
\label{T-the}
Let $J = (F_a, F_b, F_c) \subset S = k[x,y,z]$ be a q.c.i. ideal of type $(a,b,c)$, $a \leq b \leq c$. So we have an exact sequence:
\begin{equation}
\label{eq.the}
0 \to E \to \oc (c-a) \oplus \oc (c-b) \oplus \oc \to \ic _T(c) \to 0
\end{equation}
where $T \subset \Ptw$ is a closed subscheme of codimension two and where $E$ is a rank two vector bundle.

Set $t = h^0(\oc _T)$ and let $r := min\{k \mid h^0(E(k)) \neq 0\}$. Then:

(i) $$c(a + b - c - r) \leq t \leq r^2 + r(c - a -b) + ab$$\\
(ii) If $2r > a + b -c$ then:
$$t \leq r(c - a -b) + ab + r^2 -\frac{(c -a -b + 2r +1)(c - a - b +2r)}{2}.$$
\end{theorem}

\begin{proof} (i) Since $\ic _T(c)$ is generated by global sections, $T$ is contained in a complete intersection $F_a\cap G_c$ of type $(a,c)$ and we may assume $G_c = F_c$.

If $T = F_a \cap F_c$, then $t = ac, r =b-c$ and $E = \oc (c-b) \oplus \oc (-a)$ and both inequalities are satisfied (in fact they give a single equality !).

So we may assume that $T$ is linked to $\Gamma$ by $F_a \cap F_c$, where $\cG := \deg (\Gamma ) = ac -t > 0$. Now from the resolution (\ref{eq.the}), by mapping cone, taking into account that $E^*(-a) = E(b-c)$ ($E$ has rank two and $c_1(E) = c-a-b$, by Lemma \ref{L-Chern}), we get, after simplifications:
$$0 \to \oc (r-a) \to E(r) \to \ic _{\Gamma}(c-b+r) \to 0$$
Because of the Koszul syzygy: $F_b(F_a) -F_a(F_b) =0$, we have $r \leq a+b-c$. We observe that the inequality $c(a + b - c -r) \leq t$ is clearly satisfied if $r = a+b-c$. Thus from now on we may assume $r < a+b-c \leq a$. The previous exact sequence shows that $h^0(\ic _\Gamma (c-b+r) \neq 0$. So $\Gamma$ is contained in a curve of degree $c-b+r$ and in $F_a \cap F_c$. We have $c-b+r < a \leq c$. Since $h^0(\ic _\Gamma (c)) \geq 2$ and since the base locus of the linear system $|H^0(\ic _\Gamma (c))|$ has dimension zero (because $\dim (F_a \cap F_c) =0$ and $a \leq c$), we conclude that $\Gamma$ is contained in a complete intersection of type $(c-b+r, c)$. It follows that $ac - t = \cG \leq c(c-b+r)$ and we get the lower bound of (i).

By definition we have $h^0(E(r)) \neq 0$ and by minimality of $r$, a non zero section of $E(r)$ vanishes in codimension two or doesn't vanish at all. So we have an exact sequence $0 \to \oc \to E(r) \to \ic _Z(c-a-b+2r) \to 0$, where $Z$ is empty or of codimension two, of degree $c_2(E(r))$. If $Z$ is empty, then $E(r) = \oc \oplus \oc (c-a-b+2r)$. In any case $c_2(E(r)) \geq 0$. From Lemma \ref{L-Chern}, $c_2(E(r)) = r(c-a-b) + ab -t + r^2$ and we get the upper bound of (i). This concludes the proof of (i).

(ii) Consider again the previous exact sequence $0 \to \oc \to E(r) \to \ic _Z(c-a-b+2r) \to 0$. If $Z = \emptyset$, then by minimality of $r$: $c-a-b + 2r \leq 0$, against our assumption. So $Z$ is non empty and $0 = h^0(E(r-1)) = h^0(\ic _Z(c-a-b + 2r -1)$. This implies: $h^0(\oc (c-a-b +2r -1)) \leq \deg (Z) = c_2(E(r))$ and this is the desired inequality.
  
\end{proof}

\begin{remark}
\label{R-E stable}
(i) The bounds in (i) of Theorem \ref{T-the} are sharp in the sense that both inequalities are equalities if $T$ is a complete intersection $(a,c)$ (then $t = ac, r = b-c$).  

(ii) The condition $2r > a+b-c$ is equivalent to require $E$ stable ($E$ is stable if $h^0(E_{norm})=0$, where $E_{norm}$ is the twist $E(m)$ such that $-1 \leq c_1(E(m)) \leq 0$). 
\end{remark}

If $t$ reaches the upper bound in (i) or is right below, we have:

\begin{proposition}
\label{P-t max sub}
With notations as in Theorem \ref{T-the} we have:\\
(i) If $t = r^2 + r(c-a-b) + ab$, then $E = \oc (-r)\oplus \oc (c-a-b+r)$ and $J$ is an a.c.i. (and $2r \leq a+b-c$).\\
(ii) If $t = r^2 + r(c-a-b) + ab -1$, there are exact sequences:
$$0 \to \oc (r+c-a-b-2) \to \oc (-r) \oplus 2.\oc (r+c-a-b-1) \to E \to 0$$
$$0 \to 2.\oc (-c-r+1) \oplus \oc (r-a-b) \to \oc (-c-r+2) \oplus \oc (-a) \oplus \oc (-b) \oplus \oc (-c) \to \ic _T \to 0$$
and $2r \leq a+b-c+1$.
\end{proposition}

\begin{proof}
(i) If $t = r(c-a-b) + ab +r^2$, then $c_2(E(r)) =0$. Since $h^0(E(r)) \neq 0$, this implies that $E$ splits ($Z = \emptyset$ with notations as above). From the definition of $r$ we have $E = \oc (-r)\oplus \oc (c-a-b+r)$ and $r \leq a+b-c -r$.

(ii) We have $c_2(E(r))=1$ and we conclude with Lemma \ref{L-c2E(r)=1} below.
\end{proof}

\begin{lemma}
\label{L-c2E(r)=1} With notations as above assume $c_2(E(r))=1$, then we have:
$$0 \to \oc (r+c-a-b-2) \to \oc (-r) \oplus 2.\oc (r+c-a-b-1) \to E \to 0$$
$$0 \to 2.\oc (-c-r+1) \oplus \oc (r-a-b) \to \oc (-c-r+2) \oplus \oc (-a) \oplus \oc (-b) \oplus \oc (-c) \to \ic _T \to 0$$
Moreover $2r \leq a+b-c+1$.
\end{lemma}

\begin{proof} By minimality of $r$, $E(r)$ has a section vanishing along one point: $0 \to \oc \to E(r) \to \ic _p(c-a-b+2r) \to 0$. The resolution of $\ic _p$ yields a surjective morphism: $2.\oc (c-a-b+2r-1) \to \ic _p(c-a-b+2r)$. This morphism can be lifted to $E(r)$ and by completing the diagram we get:
$$0 \to \oc (c-a-b+2r-2) \to \oc \oplus 2.\oc (c-a-b+2r-1) \to E(r) \to 0$$
We observe that $T$ is not a complete intersection $(a,c)$. Indeed if it were we would have $r=c-b$ and $E(r) = \oc \oplus \oc (b-c-a)$ which has $c_2=0$. As seen in the proof of Theorem \ref{T-the}, $T$ is linked to $\Gamma$ by a complete intersection $(a,c)$ and we have $0 \to \oc (b-a-c) \to E(-c+b) \to \ic _\Gamma \to 0$. Combining with the resolution of $E$ found before we get:
$$0 \to \oc (-a+r-2) \oplus \oc (b-a-c) \to \oc (-r-c+b) \oplus 2.\oc (-a+r-1) \to \ic _\Gamma \to 0$$
Now by mapping cone, using again the complete intersection $(a,c)$, we get the resolution of $\ic _T$.

The last inequality follows from the definition of $r$ ($h^0(E(r-1))=0$). 
\end{proof}

\begin{remark}
\label{R-1 2 pts}Of course we have a similar statement if $c_2(E(r))=2$ (or if we know the minimal free resolution of $\ic _Z$).
\end{remark} 

With our notations we have $a-c \leq r \leq a-c+b$.

\begin{proposition}
\label{P-r=a-c+1}
(i) With notations as above, if $r=a-c$, then $T$ is a complete intersection of type $(a,c)$.\\
(ii) If $r = a-c+1$, then there are three cases:\\
1) $a=b$, $t=c(a-1)$ and $E(r)$ has a scetion vanishing at one point.\\
2) $a=b$, $t = c(a-1)+1$ and $E$ splits.\\
3) $b = a+1$, $t = ac$ and $E(r)$ has a section vanishing at one point.\\
When $E(r)$ has a section vanishing at one point, we get the resolutions of $E$ and $\ic _T$ from Lemma \ref{L-c2E(r)=1}.
\end{proposition}

\begin{proof} (i) We have $0 \to E(a-c) \to \oc \oplus \oc (a-b) \oplus \oc (a-c) \to \ic _T(a) \to 0$, with $h^0(E(a-c)) \neq 0$. If $b > a$, then $\lG F_a =0$ for some $\lG \neq 0$, which is absurd. Hence $a=b$. If $c > a$, then $F_a$ and $F_b$ are linearly dependent and $T = F_a \cap F_c$. If $a=b=c$, then $F_a, F_b, F_c$ are linearly dependent, one of them is a linear combination of the other two and $T$ is the complete intersection of these two curves.\\
(ii) First observe that necessarily $b \leq a+1$. We have $c_2(E(a-c+1)) = a-b-c + cb + 1 -t$. From (i) of Theorem \ref{T-the} we have: $c(a-1) \leq t \leq c(a-1)+1$, if $a=b$ and: $t=ac$ if $b=a+1$.

If $a=b$ and $t = c(a-1)$, then $c_2(E(a-c+1))=1$ and $E(r)$ has a section vanishing at one point. The same happens if $b=a+1$, $t=ac$.

If $a=b$ and $t = c(a-1)+1$, then $c_2(E(a-c+1))=0$ and $E$ splits.
\end{proof}

Now we give criteria for $J$ to be an a.c.i. ideal.

\begin{proposition}
\label{P-Jaci}
With the notations of Theorem \ref{T-the}, we have:\\
(i) $E$ splits if and only if $h^1(E(m_0)) =0$, where $m_0 = \lfloor \frac{(a+b-c-1)}{2}\rfloor$\\
(ii) If $H^0_*(E)$ has two generators of degrees $r_1, r_2$, with $r_1 + r_2 \leq a+b-c$, then $E$ splits.
\end{proposition}

\begin{proof} (i) This follows from the fact (see \cite{phe}) that a normalized rank two vector bundle, $\ec$, on $\Ptw$ splits if and only if $h^1(\ec (-1))=0$ ($\ec$ is normalized if $-1 \leq c_1(\ec ) \leq 0$).

(ii) Assume $r_1 \leq r_2$. Of course $r \leq r_1$. Consider the exact sequence $0 \to \oc \to E(r) \to \ic _Z(c-a-b+2r) \to 0$. If $Z = \emptyset$, we are done. Otherwise, since $2r \leq a+b-c$, we have $h^0(E(r))=1$. The next generator will come from a section of some twist $\ic _Z(m)$, $m > 0$. The first possibility is $0 \to \oc (a+b-c -2r+1) \to E(a+b-c-r+1) \to \ic _Z(1) \to 0$. It follows that $r_2 \geq -r +a +b -c +1 \geq -r_1 + a +b -c +1$. This implies $r_1 + r_2 > a +b -c$, contradiction. Hence $Z= \emptyset$ and $E$ splits. 
\end{proof}

\begin{remark} Condition (ii) can be useful when doing explicit computations. It was first proved in \cite{Simis} in the case $a = b =c$, by a different method.
\end{remark}

%**********************************
\section{Global Tjurina number of plane curves.}

Let $C \subset \Ptw$ be a reduced, singular curve, of degree $d$, of equation $f=0$. The partials of $f$ determine an exact sequence:

\begin{equation}
\label{eq:C}
0 \to E_C \to 3.\oc \stackrel{\partial f}{\to} \ic _\Sigma (d-1) \to 0
\end{equation}

The codimension two subscheme $\Sigma$ is the \emph{jacobian singular scheme of $C$}.

\begin{definition} The global Tjurina number of $C$, $\tau(C)$, is defined by $\tau (C) := h^0(\oc _\Sigma )$. (If no confusion can arise we will just write $\tau$ instead of $\tau (C)$).

Let $r := min \{k \mid h^0(E_C(k)) \neq 0\}$. In other words $r$ is the minimal degree of a syzygy between $f_x, f_y, f_z$.
\end{definition}

From Lemma \ref{L-Chern} we get: 

\begin{lemma}
With notations as above $c_1(E_C) = -d+1, c_2(E_C) = -\tau +(d-1)^2$ ($d = \deg (C)$, $\tau = h^0(\oc _{\Sigma})$.
\end{lemma}

Let us first recall the following well known (and easy to prove) fact:

\begin{lemma}
\label{L-partialsDep} Let $C \subset \Ptw$ be a plane curve of equation $f=0$. The partial $f_x, f_y, f_z$ are linearly dependent if and only if $C$ is a set of lines through a point.
\end{lemma}

\begin{remark}
\label{R-partialsDep} If $C$ is a set of $d$ distinct lines through a point, then $\Sigma$ is the complete intersection of two partials and $\tau (C)=(d-1)^2$. Indeed since $C$ is reduced there exist two linearly independent partials and they don't share any common component. Moreover in this case it is easy to see that $E_C = \oc \oplus \oc (-d+1)$ (combine the exact sequence $0 \to \oc (-2d+2) \to 2.\oc (-d+1) \to \ic _\Sigma \to 0$ with the exact sequence defining $E_C$).
\end{remark}

Clearly we have $r \leq d-1$ (Koszul relations). The case $r=0$ is settled by Lemma \ref{L-partialsDep}, hence in the sequel we will assume $1 \leq r \leq d-1$. 

The following definition goes back to Saito \cite{Saito}:

\begin{definition} With notations as above, the divisor $C \subset \Ptw$ is said to be \emph{free} if $E_C = \oc (-a)\oplus \oc (-b)$. In this case $(a,b), a \leq b,$ is called the \emph{exponent} of $C$.
\end{definition}

\begin{remark}
\label{R-free} (i) If $E_C$ splits then, from the definition of $r$, $E_C = \oc (-r)\oplus \oc (-d+1+r)$. In this case $c_2(E_C) = (d-1)^2 -\tau = -r(-d+1-r)$. It follows that $\tau = (d-1)(d-1-r)+r^2$. If $-r(-d+1-r)=0$, then $\tau = (d-1)^2$ and $\Sigma$ is a complete intersection $(d-1, d-1)$. This implies $r=0$ hence $C$ is a set of lines through a point.

(ii) If $a,b$ are two integers such that $a+b = d-1$, then the maximal value of $ab$ is $(d-1)^2/4$ if $d$ is odd and $(d-2)d/4$ if $d$ is even. It follows that $(d-1)^2 - \tau \leq (d-1)^2/4$ if $d$ is odd, i.e. $\tau \geq 3(d-1)^2/4$; if $d$ is even we get: $\tau \geq 1 + \frac{3d(d-2)}{4}$. This has been already observed in \cite{Di-MaxTju}. So free curves have a \emph{big} global Tjurina number (see also the proof of Proposition \ref{P-last}).

(iii) Since a stable vector bundle is indecomposable, if $C$ is free $2r + 1 \leq d$ and the exponent is $(r, d-1-r)$.

(iv) An important fact about free curves is that they exist! Indeed, as proved in \cite{DS-expFree}, for every $d \geq 3$ and any $r$, $1 \leq r < d/2$, there exists a free curve with exponent $(r, d-1-r)$.    
\end{remark}

\section{Bounds on $\tau (C)$ and a theorem of du Plessis-Wall.}

It is natural to ask for a bound of $\tau$ in function of $d$ and other invariants of $C$. The following result has been proved by du Plessis and Wall (see Theorem 3.2 of \cite{dP-Wall}). 

\begin{theorem} \emph{(du Plessis-Wall)}\\
\label{T-dP-Wall}
Let $C \subset \Ptw$ be a reduced, singular curve of degree $d$.\\ (i) Then:
\begin{equation}
(d-1)(d-r-1) \leq \tau \leq (d-1)(d-r-1) + r^2
\end{equation}
(ii) If, moreover, $2r+1 > d$, then:
\begin{equation}
\tau \leq (d-1)(d-r-1)+r^2 - \frac{1}{2}(2r +1 -d)(2r +2-d)
\end{equation}
\end{theorem}

\begin{proof} Just put $a=b=c =d-1$ and $t=\tau$ in Theorem \ref{T-the}.
\end{proof}

Let us see that the bounds in the first part of the theorem are sharp. 

\begin{lemma}
\label{L- t=(d-1)(d-1-r)}
For every $d \geq 2$, there exists a curve of degree $d$ with $\tau = (d-1)(d-1-r)$.
\end{lemma}

\begin{proof} Let $C=X \cup L$ where $X$ is a smooth curve of degree $d-1$ intersecting the line $L$ transversally at $d-1$ distinct points. Clearly $\Sigma = X\cap L$ and $\tau = d-1$. It remains to show that $r=d-2$. We may assume that $L$ has equation $x=0$. Let $g =0$ be an equation of $X$, so that $C$ has equation $f=xg$. Since $\Sigma$ is the complete intersection $X\cap L$, we have $0 \to \oc (-d) \to \oc (-d+1)\oplus \oc (-1) \to \ic _\Sigma \to 0$. We have a commutative diagram:
$$\begin{array}{ccccccccc}
0 & \to & E_C(-d+1) & \to & 3.\oc (-d+1)& \to &\ic _\Sigma &\to & 0\\
 &  & \downarrow \psi &  & \downarrow \fG&  &|| & & \\
0 & \to & \oc (-d) & \to & \oc (-d+1)\oplus \oc (-1)& \to &\ic _\Sigma &\to & 0\end{array}$$
where $\fG$ is given by $M=\left( \begin{array}{ccc}
1 &0 &0\\
g_x & g_y & g_z\end{array} \right)$, the expressions of $f_x, f_y, f_z$ in function of $x,g$. We have $Ker(\fG )=Ker(\psi )=K$. We observe that $K$ is reflexive ($E_C(-d+1)$ is locally free and $Im(\fG )$ is torsion free), hence locally free. Since $\fG$, hence also $\psi$ are non zero and since $Im(\psi )$ is torsion free, we conclude that $K$ has rank one, say $K = \oc (-c)$. Clearly $c-d+1$ is the least degree of a relation between $g_y, g_z$ (look at $Ker(\fG )$ and $M$). If we take $g= x^{d-1}+h(y,z)$, $g_y=h_y, g_z=h_z$, and since $\Sigma$ is smooth, we get $c-d+1 =d-2$, hence $c=2d-3$. 

The locus where $\fG$ doesn't have rank two is defined by the $2\times 2$ minors of $M$, it is the complete intersection $(g_y)_0 \cap (g_z)_0$. Since $Coker(\psi )=Coker(\fG )$, it is also the locus where $\psi$ is not vector-bundle surjective, it is the locus where the section $\oc \to E_C(c-d+1)$ vanishes. In particular it has codimension two and we have an exact sequence $0 \to \oc \to E_C(-d+1+c) \to \ic _Z(2c-3(d-1)) \to 0$. Since $2c -3(d-1) = d-3$ and since $Z$ is a complete intersection $(d-2, d-2)$, we get $h^0(\ic _Z(2c-3(d-1))=0$, hence $h^0(E_C(-d+1+c))=1$ and this shows that $r = -d+1+c = d-2$.
\end{proof}

\begin{proposition}
\label{P- bds sharp}
For every $d \geq 3$, there exist $r$ and $\tau$ satisfying the bounds in (i) of Theorem \ref{T-dP-Wall}. 
\end{proposition}

\begin{proof} For the bound $(d-1)(d-r-1) \leq \tau$, this follows from Lemma \ref{L- t=(d-1)(d-1-r)}. For the bound $\tau \leq (d-1)(d-r-1) + r^2$, this follows from \cite{DS-expFree} (see Remark \ref{R-free}).
\end{proof}

\begin{remark} For the lower bound we have examples only with $\tau = d-1, r= d-2$. I don't know if other values are possible.
\end{remark}

%**************************************

\section{Further topics.}

We have the following definition (see \cite{DS-nearly}):

\begin{definition} The curve $C$ is said to be \emph{nearly free} if  we have:
$$ 0 \to S(-d+r-1) \to S(-r) \oplus 2.S(-d+r) \to H^0_*(E_C) \to 0.$$
\end{definition}

It can be shown (see \cite{Di-MaxTju} or also \cite{phe}) that this is equivalent to: $E_C(r)$ has a section vanishing at one point (so we are near to the free case where the section doesn't vanish at all).

If $\tau$ reaches the bound in Theorem \ref{T-dP-Wall} (i) or is just below, we have:

\begin{proposition}
\label{L-tau reaches bds}
If $\tau = (d-1)(d-r-1) + r^2$, then $C$ is free.\\
If $\tau = (d-1)(d-r-1) + r^2 -1$, then $C$ is nearly free.
\end{proposition}

\begin{proof} Apply Proposition \ref{P-t max sub} with $a=b=c = d-1$ and $t = \tau$.
\end{proof}

The case $r=0$ is completely settled (Lemma \ref{L-partialsDep}) and it is natural to investigate the cases where $r$ is small.

\begin{proposition}
\label{r=1}
If $r=1$ then: $\tau = (d-1)(d-2)+1$ and $C$ is free, or, $\tau = (d-1)(d-2)$ and $C$ is nearly free.
\end{proposition}

\begin{proof} This follows from Proposition \ref{P-r=a-c+1} (ii). 
\end{proof}

\begin{remark} The last two propositions are already known (except maybe the resolution of $\ic _\Sigma$), see \cite{Di-MaxTju}.
\end{remark}

Let us conclude with the following:

\begin{proposition}
\label{P-last}
Assume $C$ is not a set of lines through a point. Then:\\
(i) $\tau \leq d^2 -3d +3$\\
(ii) Assume $d > 7$. If $\tau > d^2 -4d +5$, then:\\
(1) $\tau = d^2 -3d +3$, $r=1$ and $C$ is free.\\
(2) $\tau = d^2 -3d +2$, $r=1$ and $C$ is nearly free.\\
(3) $\tau = d^2 -4d +7$, $r=2$ and $C$ is free.\\
(4) $\tau = d^2 -4d +6$, $r=2$ and $C$ is nearly free.
\end{proposition}

\begin{proof} By Theorem \ref{T-dP-Wall} we have: $\tau \leq \fG (r) = r^2 -r(d-1) + (d-1)^2$, if $0 \leq r \leq (d-1)/2$ and $\tau \leq \psi (r) = -r^2 + r(d-2) + d(d-1)/2$, for $(d-1)/2 < r \leq d-1$.

The function $\fG (r)$ on $I = [0, (d-1)/2]$ is decreasing, we have $\fG (0) = (d-1)^2$, $\fG (1) = d^2 -3d +3$ and $\fG (2) = d^2 -4d +7$. The function $\psi (r)$ reaches its maximal value for $(d-2)/2$ and is decreasing on $[(d-1)/2, d-1]$. It follows that $\psi (r) < \psi (\frac{d-1}{2}) = 3(d-1)^2/4$ for $r \in I_s = ](d-1)/2, d-1]$.

Observe that $\fG (\frac{d-1}{2}) = \psi (\frac{d-1}{2}) = 3(d-1)^2/4$ (cp. Remark \ref{R-free}). 

Since $d^2 -3d + 3 \geq 3(d-1)^2/4$ for every $d \geq 2$, if $\tau > d^2 -3d +3$, we have $r \leq (d-1)/2$ and then $r=0$, which is impossible under our assumption. Hence $\tau \leq d^2 -3d +3$ and this proves (i).

(ii) Observe that $d^2 -4d + 6 \geq 3(d-1)^2/4$ if $d \geq 7$. So if $\tau \geq d^2 -4d +6$, we must have $r \leq (d-1)/2$. Moreover since $d^2 -4d +6 > \fG (3)$, if $d > 7$, we have $1 \leq r \leq 2$. If $r=1$ we conclude with Proposition \ref{r=1}, obtaining cases (1) and (2).

If $r=2$, by Theorem \ref{T-dP-Wall}: $(d-1)(d-3) \leq \tau \leq (d-1)(d-3) + 4$. We have $c_2(E_C(2)) = (d-1)(d-3) + 4 -\tau$. Hence if $\tau = d^2 -4d + 7$, $c_2(E_C(2)) =0$ and $E_C$ splits. If $\tau = d^2 -4d + 6$, then $c_2(E_C(2)) = 1$ hence $E_C(r)$ has a section vanishing along one point and $C$ is nearly free. These are cases (3) and (4).  
\end{proof}

\begin{remark} (i) The existence of free and nearly free curves with invariants as in the previous proposition follows from \cite{DS-expFree}.

(ii) For fixed $d$ the possible values of $\tau$ seem sparse. 

Let $G(d,3) = 1 +\frac{d(d-3) - 2r(3-r)}{6}$, where $d+r \equiv 0 \pmod{3}, 0 \leq r < 3$. Then it is known (\cite{GP}) that for every $g$, $0 \leq g \leq G(d,3)$, there exists a smooth, irreducible, non degenerated curve $X \subset \Pt$, of degree $d$, genus $g$. A general projection of $X$ in $\Ptw$ is a curve of degree $d$ with $\tau = (d-1)(d-2)/2 - G(d,3)$ nodes. We conclude that for $(d-1)(d-2)/2 - G(d,3) \leq \tau \leq (d-1)(d-2)/2$, there exists a curve, $C$, of degree $d$ with $\tau (C) =\tau$.

\end{remark}

%*****************************************
%*****************************************

%\printindex


\begin{thebibliography}{Tjurina}

\bibitem{BeoE1} Beorchia, V.-Ellia, Ph.: \emph{Normal bundle and complete intersections}, Rend. Sem. Mat. Univers. Politecn. Torino, vol. 48, 4, 553-562 (1990)

\bibitem{BeoE2} Beorchia, V.-Ellia, Ph.: \emph{On the equations defining quasi complete intersections space curves}, Arch. Math., {\bf 70}, 244-249 (1998)

\bibitem{Di-MaxTju} Dimca, A.: \emph{Freeness versus maximal global Tjurina number for plane curves},  Math. Proc. Cambridge Philos. Soc. {\bf 163}, n° 1, 161-172, (2017)

\bibitem{DS-expFree} Dimca, A.-Sticlaru, G.: \emph{On the exponents of free and nearly free projective plane curves},  Rev. Mat. Complut., {\bf 30}, n° 2, 259-268 (2017)

\bibitem{DS-nearly} Dimca, A.-Sticlaru, G.: \emph{Nearly free divisors and rational cuspidal curve}, arXiv:1505.00666

\bibitem{dP-Wall} du Plessis A.A.- Wall C.T.C.: \emph{Application of the theory of the discriminant to highly singular plane curves}, Math. Proc. Camb. Phil. Soc., 126, 259-266 (1999)

\bibitem{phe} Ellia, Ph.: \emph{On the cohomology of rank two vector bundles on $\Ptw$ and a theorem of Chiantini and Valabrega}, preprint arXiv (2019)

\bibitem{GP} Gruson, L.-Peskine, Ch.: \emph{Genre des courbes de l'espace projectif, II}, Ann. Sci. \'Ecole Norm. Sup. (4), {\bf 15}, n° 3, 401-418 (1982)



\bibitem{Ha} Hartshorne, R.:\textit{Stable reflexive sheaves}, Math. Ann.,{\bf 254}, 121-176 (1980) 


\bibitem{Saito} Saito, K.: \emph{Theory of logarithmic differentials forms and logarithmic vector fields}, J. Fac. Sci. Univ. Tokyo Sect Math., {\bf 27}, 265-291 (1980)

\bibitem{Simis} Simis, A.-Toh{\u{a}}neau, S.O.: \emph{Homology of homogeneous divisors}, Israel J. of Math., {\bf 200}, 449-487 (2014)








\end{thebibliography}
\end{document}